\documentclass{article}
\usepackage{amsmath,amssymb,amsfonts}
\usepackage{paralist}
\usepackage{mathrsfs}
\usepackage{hyperref}

\newcommand{\bp}{\boldsymbol{p}}
\newcommand{\bq}{\boldsymbol{q}}
\newcommand{\bx}{\boldsymbol{x}}

\newcommand{\bw}{\boldsymbol{w}}

\newcommand{\coneN}{\mathbb{K}^{N}_{+}}
\newcommand{\intconeN}{\overset{\circ}{\mathbb{K}}{\vphantom{\mathbb{K}}}^{N}_{+}}
\newcommand{\simpN}{\Sigma_{N}}
\newcommand{\intsimpN}{\overset{\circ}{\Sigma}_{N}}

\newtheorem{corollary}{Corollary}
\newtheorem{theorem}{Theorem}
\newtheorem{remark}{Remark}
\newtheorem{definition}{Definition}

\begin{document}
\date{}

\title{Asymptotic behaviour of random Markov chains with
tridiagonal generators}

\author{\textsc{P. E.~Kloeden}\thanks{Partially supported the DFG grant
KL~1203/7-1, the  spanish Ministerio de Ciencia e Innovaci\'{o}n
project MTM2011-22411, the Consejer\'{\i}a de Innovaci\'{o}n,
Ciencia y Empresa (Junta de Andaluc\'{\i}a) under the Ayuda
2009/FQM314 and
the Proyecto de Excelencia P07-FQM-02468.}\\
\small\textsl{Institut f\"ur  Mathematik, Goethe Universit\"at},\\
\small\textsl{D-60054 Frankfurt am Main, Germany}\\
\small E-mail:~\texttt{kloeden@math.uni-frankfurt.de}\\[5mm]
\textsc{V. S.~Kozyakin}\thanks{Partially supported by the
Russian Foundation for Basic Research, project no. 10-01-93112.}\\
\small\textsl{Institute for Information Transmission
Problems}\\\small\textsl{Russian Academy of Sciences}\\
\small\textsl{Bolshoj Karetny lane, 19, 101447 Moscow, Russia}\\
\small E-mail:~\texttt{kozyakin@iitp.ru}}

\maketitle


\vspace*{0.3cm}

\abstract Continuous-time discrete-state random Markov chains
generated by a random linear differential equation with  a random
tridiagonal matrix are shown to have a random attractor consisting
of singleton subsets, essentially a random path, in the simplex of
probability vectors. The proof uses comparison theorems for
Carath\'eodory random differential equations  and the fact that the
linear cocycle generated by the Markov chain is a uniformly
contractive mapping of the positive cone into itself with respect
to the Hilbert projective metric.  It does not involve
probabilistic properties of the sample path   and is thus equally
valid in the nonautonomous deterministic context of Markov chains
with, say, periodically varying transitions probabilities, in which
case the attractor is a periodic path.

\vspace*{0.1cm}

\noindent\textbf{AMS Subject Classification:}\quad Primary AMS
Subject Classification:   34F05, 37H10, 60H25, 60J10  Secondary
15B48, 15B51, 15B52. \newline

\noindent\textbf{Key words:}\quad Random differential equations,
random Markov chains, positive cones, linear cocycles, uniformly
contracting cocycles,  random attractors \newline

\vspace*{0.2cm}

\section{Introduction}\label{intro}

The  simplex $\Sigma_N$ in $\mathbb{R}^N$ defined by
\[
\Sigma_N = \bigg\{ \bp = (p_1,\dots,p_N)^T :~ \sum_{j=1}^ N p_j = 1, ~ p_1,\ldots,p_N \in[0,1] \bigg\},
\]
is invariant under the dynamics of  the system of ordinary
differential equations
\begin{equation}\label{DDE}
\frac{d\bp}{dt} = Q \bp
\end{equation}
in $\mathbb{R}^N$,  where the   $N\times N$ matrix $Q$ satisfies
\[
\boldsymbol{1}_N^T Q \equiv \boldsymbol{0}.
\]
Here $\boldsymbol{1}_N$ is the column vector in $\mathbb{R}^N$ with
all  components   equal to $1$.

The system \eqref{DDE} restricted to $\Sigma_N$  generates a
continuous-time finite-state Markov chain with a generator $Q$.
Such Markov chains are common in the  biological sciences with
generators given by  an  $N\times N$ tridiagonal matrix
\begin{equation}\label{E-defQ}
Q = \left[\begin{smallmatrix}
-q_1 & q_2 & & & & & \bigcirc\\
\hphantom{-}q_1 & -(q_2+q_3) & q_4 & & & & \\
 & \ddots & \ddots & \ddots & \ddots & \ddots & \\
 & & & & q_{2N-5} & -(q_{2N-4}+q_{2N-3}) & \hphantom{-}q_{2N-2} \\
\bigcirc & & & & & q_{2N-3} & - q_{2N-2}
\end{smallmatrix}\right],
\end{equation}
with positive  off-diagonal components $\{q_{i}\}$, see
\cite{AllenL:10,AEH:09,WodKom:05}.    The components of $Q$ may
depend on time or be random, or both, i.e., depend on a stochastic
process, in which case $Q$ will be written $Q(\theta_t \omega)$
(the undefined terms will be explained later).

The asymptotic behaviour of the discrete-time version  of such
Markov chains were investigated by the authors in
\cite{KloKoz:DCDSB11}. Their transition matrices are of the form
$I_N + Q(t_n) \Delta_n$, where  $I_N$ is the $N\times N$   identity
matrix, $Q(t_n)$ is value of the generator matrix  $Q$ at time
$t_n$, i.e.,  $Q(t_n) = Q(\theta_{t_n} \omega)$,  and $\Delta_n$ is
a positive parameter. This is just the Euler scheme for the
differential equation  \eqref{DDE} on $\Sigma_N$ with the time step
$\Delta_n$. It was shown in \cite{KloKoz:DCDSB11} that these
discrete-time Markov chains generate a linear random dynamical
system which has a random attractor consisting of singleton
components sets, essentially a random process of probability
vectors in $\Sigma_N$.

In this paper the corresponding result will be established for
contin-uous-time finite-state Markov chains generated by the
differential equation \eqref{DDE}  in $\Sigma_N$, where $Q$ is  an
$N\times N$ tridiagonal matrix of the form  \eqref{E-defQ} with
components driven by a stochastic process. This will be proved
directly. An indirect proof using the convergence of the  random
attractors of the Euler scheme  requires them to have a uniform
rate of attraction. This is the inverse problem of what is  usually
discussed in numerical dynamics, where  it is shown that the
numerical scheme preserves the properties of continuous-time
systems generated by  the  differential equation, see
\cite{KloLor:SIAMJNA86,SH}.

The paper is structured as follows.  Basic properties of
differential equations satisfying the Carath\'{e}odory property are
recalled in Section~\ref{D-ineq}. In particular, in
Subsection~\ref{D-dynamics},
 properties of deterministic linear differential
equations of the type \eqref{DDE} are investigated, while  in
Subsection~\ref{D-rand}
 some notions on related random differential
equations (RDE) of the type \eqref{DDE} are recalled. Finally, in
Subsection~\ref{D-rancos} it is shown that that the related RDE
preserves the positivity of  initial values and in
Section~\ref{S-random} the main result of the paper is presented, a
theorem on the existence of a random attractor consisting of
singleton component subsets in the system under consideration.

\section{Carath\'{e}odory differential equations }\label{D-ineq}

 The linear ordinary  differential equation  \eqref{DDE} with the matrix $Q$ depending on a stochastic process
is a differential equation of the Carath\'{e}odory type. Some basic
properties of such equations are recalled here along with some
comparison theorems,  which will be used  to establish the
positivity of solutions of the resulting random differential
equation.

An ordinary  differential equation on $\mathbb{R}^{N}$,
\begin{equation}\label{E-NDE}
    \frac{d\bx}{dt}=f(t,\bx),
\end{equation}
is called a Carath\'{e}odory  differential equations if the
vector-field $f(t,\bx)$ satisfies the \emph{Carath\'{e}odory
conditions}:
\begin{enumerate}[C1:]
  \item for every fixed $t$, the function $\bx \mapsto$
      $f(t,\bx)$ is continuous;
  \item for every fixed $\bx$, the function $t \mapsto$
      $f(t,\bx)$ is measurable in $t$ and there exist a
      function $m_{\bx}(t)$, which is Lebesgue integrable on
      every bounded subinterval, such that $\|f(t,x)\| \le$
      $m_{\bx}(t)$.
\end{enumerate}

A solution of differential equation \eqref{E-NDE} is  a
vector-valued  function $\bx(t)$ which is absolutely continuous on
some interval and satisfies \eqref{E-NDE} almost everywhere on this
interval. The existence of solutions of \eqref{E-NDE} on the whole
real axis $\mathbb{R}$  as well as their uniqueness follows from
the Carath\'{e}odory conditions, if, in addition, the  function
$f(t,\bx)$ is defined for all $(t,\bx) \in$
$\mathbb{R}^{1}\times\mathbb{R}^{N}$ and satisfies a local
Lipschitz condition in $\bx$ in a neighbourhood of every initial
point $(t_{0},\bx_{0})$.

A solution $\bx(t) = (x_{1}(t),x_{2}(t),\ldots,x_{N}(t))^\top$ of
\eqref{E-NDE} is called \emph{positive} (\emph{strongly positive})
if
\[
x_{i}(t)\ge 0~ (> 0)\quad\textrm{for all   }  ~t~\textrm{and}~i=1,2,\ldots,N.
\]

The next condition  guarantees the  positivity of solutions of
\eqref{E-NDE}, see, e.g.,
\cite{Szarski:65,Kras:OpTrans:e,Walter:91,Smith95}.  A  function
$f=(f_{1},f_{2},\ldots,f_{N})^\top
:\mathbb{R}^{1}\times\mathbb{R}^{N} \rightarrow \mathbb{R}^{N}$ is
called \emph{quasipositive}, or \emph{off-diagonal positive}, if,
for each $i = 1,2,\ldots,N$,
\[
f_{i}(t,x_{1},\ldots,x_{i-1},0,x_{i+1},\ldots,x_{N})\ge 0,
\]
whenever  $x_{j} \ge 0$ for $j \neq i$. A  function $f =$
$(f_{1},f_{2},\ldots,f_{N})^\top$ is called \emph{strongly
quasipositive}, or \emph{strongly off-diagonal positive}, if, for
each $i = 1,2,\ldots,N$,
\[
f_{i}(t,x_{1},\ldots,x_{i-1},0,x_{i+1},\ldots,x_{N})> 0,
\]
whenever $x_{j}\ge0$ for all $j$ and $\sum_{j}x_{j}>0$.

The proof of the following  theorem for the case that $f(t,\bx)$ is
continuous in $t$ and $\bx$  can be found, e.g., in
\cite[Lemma~4.1]{Kras:OpTrans:e}.  It was noted in
\cite{Szarski:65} that similar statements   are also valid when
$f(t,\bx)$ satisfies the Carath\'{e}odory conditions.

\begin{theorem}\label{T-diffpos} If the  vector-field   $f$
of \eqref{E-NDE} is quasipositive, then the solution $\bx(t)$ of
\eqref{E-NDE} satisfying the initial condition $\bx(0) = \bx_{0}=
(x_{0,1},x_{0,2},\ldots,x_{0,N})$ is positive for $t \ge 0$
whenever $x_{0,i} \ge 0$ for $i=1,2,\ldots,N$.

If  the vector-field   $f$ of \eqref{E-NDE} is  strongly
quasipositive, then   the solution $\bx(t)$ of \eqref{E-NDE}
satisfying the initial condition $\bx(0) = \bx_{0}
=(x_{0,1},x_{0,2},\ldots,x_{0,N})$ is strongly positive for $t$
$\ge 0$ whenever $x_{0,i} > 0$ for $i=1,2,\ldots,N$.
\end{theorem}

\subsection{Deterministic linear differential equations}\label{D-dynamics}

Unfortunately  the requirement that a vector-field function be
quasipositive is rather restrictive and does not apply to a
deterministic linear  differential equation
\begin{equation}\label{DDE1}
\frac{d\bp}{dt} = Q \bp, \quad \bp \in \mathbb{R}^N
\end{equation}
with a tridiagonal  matrix $Q$ given by \eqref{E-defQ}.

There is, however,   another useful notion that can be used in
combination with it. An $N\times N$  matrix $Q$ with non-negative
off-diagonal elements is said to  have a \emph{path of
nonsingularity} \cite{Kras:OpTrans:e} if there is a set of indices
$i_{1},i_{2},\ldots, i_{n}$ with $i_{n} = i_{1}$ such that all the
elements $q_{i_{j},i_{j+1}}$ are strictly positive and this set of
indices contains all the numbers $1$, $2$, $\ldots$, $N$. In
particular,  the matrix $Q$ defined by \eqref{E-defQ} has the  path
of nonsingularity:
\[
\{i_{1},i_{2},\ldots, i_{2N-1}\}=\{1,2,\ldots,N-1,N,N-1,\ldots,2,1\}.
\]

Recall that a set $K$ in a Banach space is called a \emph{cone} if
it is convex, closed with  $tK \subseteq K$ for any real $t \ge 0$
and $K\cap -K = \{0\}$, see, e.g. \cite{KLS:PosLinSys:e}. Fix a
norm $\|\cdot\|$ in $\mathbb{R}^{N}$ and denote by $\coneN$ the
cone of elements $x = (x_{1},x_{2},\ldots,x_{N})^\top \in
\mathbb{R}^{N}$ with nonnegative components and by $\intconeN$ the
interior of $\coneN$, which is clearly non-empty. In addition,
denote by  $\mathbb{S}^{N}$ the unit ball in the norm $\|\cdot\|$.

\begin{theorem}\label{T-linDU}
Suppose that  the matrix $Q$ in \eqref{DDE1} with nonnegative
off-dia\-go\-nal entries has a path of nonsingularity. Then, for
any non-zero initial condition $\bp(0) = \bp_{0} \in \coneN$, the
solution $\bp(t,\bp_{0})$ is strongly positive for all  $t > 0$.

Moreover, for every  bounded interval $[T_{1},T_{2}] \subset$
$(0,\infty)$, there is a number $\alpha(T_{1},T_{2}) > 0$ such that
\begin{equation}\label{E-incone}
\bp(t;\bp_{0}) + \alpha(T_{1},T_{2})\|\bp_{0}\| \, \mathbb{S}^{N}
\subseteq \mathbb{K}^{N}_{+}
\end{equation}
 for all $\bp_{0} \in \coneN\setminus\{0\}$ and $t \in [T_{1},T_{2}]$.
 \end{theorem}
\begin{proof}
The first part of Theorem~\ref{T-linDU} follows from
\cite[Theorem~4.7]{Kras:OpTrans:e}, while  the second part is clear
modulo the first part.
\end{proof}

\begin{remark}\label{R-incone}
The inclusion \eqref{E-incone} means that the solution operator
$\bp_{0} \mapsto \bp(t,\bp_{0})$ maps the set
$\coneN\setminus\{0\}$ into its interior $\intconeN$ for every $t$
$> 0$.
\end{remark}

\begin{remark}\label{R-incone2}
If the matrix $Q$ in \eqref{DDE1} satisfies the conditions of
Theorem~\ref{T-linDU} as well as  $\boldsymbol{1}_N^T Q =
\boldsymbol{0}$,  then $\bp(t;\Sigma_N) \subseteq \Sigma_N$ for all
$t \geq 0$. Moreover,  the simplex $\Sigma_N$ is mapped by
$\bp(t;\cdot)$  into its interior $\intsimpN$ for every $t > 0$.
\end{remark}

\subsection{Random linear differential equations}\label{D-rand}

Let $(\Omega,\mathcal{F},\mathbb{P})$ be a probability space. In
the theory of random dynamical systems
\cite{ArnoldL:98,Chueshov:02} the noise process that drives the
dynamics is given in terms of   an abstract  metric dynamical
system.

\begin{definition}\label{Def-MDS}
A metric dynamical system (MDS) $\theta \equiv$
$\left\{\theta_{t}\right\}_{t\in\mathbb{R}}$ on a probability space
$(\Omega,\mathcal{F},\mathbb{P})$ is  a family of transformations
$\theta_{t} : \Omega \rightarrow \Omega$, $t \in \mathbb{R}$, such
that
\begin{enumerate}[\rm 1.]
\item it is one-parameter group, i.e.,
\[
\theta_{0} = \mbox{id}_{\Omega}, \quad \theta_{t}\circ\theta_{s}  = \theta_{t+s}
\quad\textrm{for all}\quad t, s\in\mathbb{R};
\]

\item $(t,\omega)\mapsto \theta_{t}\omega$ is jointly
    smeasurable;

\item $\theta_{t}\mathbb{P} = \mathbb{P}$ for all $t \in$
    $\mathbb{R}$, i.e., $\mathbb{P}(\theta_{t}B) =$
    $\mathbb{P}(B)$ for all $B \in \mathcal{F}$ and $t \in$
    $\mathbb{R}$.
\end{enumerate}
\end{definition}

Consider the random linear differential equation
\cite{ArnoldL:98,Chueshov:02}
\begin{equation}\label{RDE}
\frac{d\bp}{dt} = Q(\theta_{t}\omega) \bp,\quad
\bp(0)=\bp_{0}\in\mathbb{R}^{N},
\end{equation}
driven by the metric dynamical system $\theta$  on a probability
space
 $(\Omega,\mathcal{F},\mathbb{P})$, where the coefficients $q_{j} :$
$\Omega \to \mathbb{R}$ in the $N \times
 N$-matrix $Q$   are locally
bounded $\mathcal{F}$-measurable functions.

More precisely,   assume that there exists a random variable $C$
with $C(\omega) > 0$   such that
\begin{equation}\label{E-C1}
\int_{\alpha}^{\alpha+1}C(\theta_{t}\omega)\,dt  < \infty
\quad\textrm{for all}\quad  \alpha\in\mathbb{R},
\end{equation}
and
\begin{equation}\label{E-C2}
|q_{i}(\omega)| \le C(\omega) \quad\textrm{for}\quad  i =1,2,\dots,2N-2
\end{equation}
for all $\omega \in \Omega$.  It is to be  emphasized that
assumptions \eqref{E-C1} and \eqref{E-C2} are stated here for all
$\omega \in \Omega$. As explained in \cite[p.~56]{Chueshov:02} this
does not restrict generality.

\begin{remark}\label{R-1}
Conditions \eqref{E-C1} and \eqref{E-C2} are satisfied if
\[
|q_{i}(\omega)|\le C_{*} \quad\textrm{for all} ~ i ~\textrm{and}~\omega\in\Omega
\]
for some constant $C_{*} > 0$.
\end{remark}

The  random differential equations \eqref{RDE} is thus a
Carath\'eodory ordinary differential equation for each sample path.
For each $\omega \in \Omega$ the vector-field mapping $f(t,\bp) =
Q(\theta_{t}\omega) \bp$ satisfies the Carath\'eodory conditions
\eqref{E-C1} and \eqref{E-C2} and is linear and thus globally
Lipschitz in $\bp$. It follows that  the initial value problem
\eqref{RDE} has a unique global solution $\bp(t,\omega)\equiv
\bp(t,\omega;\bp_{0})$ for any initial value
$\bp_{0}\in\mathbb{R}^{N}$, where this solution exists for all $t$
$\geq 0$ (see \cite{Chueshov:02} for more details explicitly in
terms of the random differential equations). Moreover, this
solution depends continuously on $\bp_{0}$ for each
$\omega\in\Omega$.

It will now be shown that  RDE \eqref{RDE}  preserves  the
positivity of  initial values.

\subsection{Positivity of solutions}\label{D-rancos}

The results from the previous sections will now be applied to the
random linear differential equation (RDE) in $\mathbb{R}^{N}$
\begin{equation}\label{RDE1}
\frac{d\bp}{dt} = Q(\theta_{t}\omega) \bp,\quad
\bp(0)=\bp_{0}\in\mathbb{R}^{N},
\end{equation}
driven by a metric dynamical system $\theta$, where the matrix
$Q(\omega)$ is of the tridiagonal form \eqref{E-defQ} and satisfies
\begin{equation}\label{E-oneQ2zero2}
\boldsymbol{1}_N^T Q(\omega) \equiv \boldsymbol{0} \quad \mbox{for all  } \omega\in\Omega.
\end{equation}

Suppose further that the coefficients in the matrix $Q(\omega)$ are
locally bounded $\mathcal{F}$-measurable functions satisfying
uniform   bounds
\begin{equation}\label{E-qbound}
0<c_{*}\le q_{i}(\omega)\le C_{*}<\infty,\quad i=1,2,\ldots,2N-2,~\omega\in\Omega.
\end{equation}
Then, by Remark~\ref{R-1},  $Q(\omega)$ satisfies the  conditions
\eqref{E-C1} and  \eqref{E-C2}.

Define (a deterministic) tridiagonal matrix
\[
\Bar{Q} = \left[\begin{array}{ccccccc}
\Tilde{q}_1 & \Bar{q}_2 & & & & & \bigcirc\\
\Bar{q}_1 & \Tilde{q}_2 & \Bar{q}_4 & & & & \\
 & \ddots & \ddots & \ddots & \ddots & \ddots & \\
 & & & & \Bar{q}_{2N-5} & \Tilde{q}_{N-1} & \Bar{q}_{2N-2} \\
\bigcirc & & & & & \Bar{q}_{2N-3} & \Tilde{q}_{N}
\end{array}\right]
\]
with the off-diagonal elements $\Bar{q}_{i}$
\[
\Bar{q}_{i}=\inf_{\omega\in\Omega} q_{i}(\omega),\quad i=1,2,\ldots,2N-2,
\]
and the diagonal elements $\Tilde{q}_{i}$ satisfying
\begin{alignat*}{3}
\Tilde{q}_{1}&=\inf_{\omega\in\Omega} \{-q_{1}(\omega)\},\\
\Tilde{q}_{i}&=\inf_{\omega\in\Omega} \{-(q_{2i-2}(\omega)+q_{2i-1}(\omega))\},\quad &i&=2,\ldots,2N-1,\\
\Tilde{q}_{N}&=\inf_{\omega\in\Omega} \{-q_{2N-2}(\omega)\}.
\end{alignat*}
Then
\[
Q(\omega)\ge \Bar{Q},\quad \mbox{for all  } \omega\in\Omega,
\]
where the inequality between the matrices is interpreted
componentwise. The off-diagonal elements  $\Tilde{q}_{i}$ of the
matrix $\Bar{Q}$ are  strictly positive since
\begin{equation}\label{E-qbarpos}
0 < c_{*}\le \Bar{q}_{i},\quad i=1,2,\ldots,2N-2.
\end{equation}

Consider   the differential equation
\begin{equation}\label{E-QbarDU}
\frac{d\bq}{dt} = \Bar{Q}\bq
\end{equation}
and denote  the solution of this equation with  the initial
condition $\bq(0) = \bq_{0}\in\mathbb{R}^{N}$ by $\bq(t;\bq_{0})$.

\begin{remark}\label{R-bpprop}
 It follows from \eqref{E-qbarpos}  that $\Bar{Q}$ is a matrix with
non-negative off-diagonal elements that has  a path of
nonsingularity. Then by Theorem~\ref{T-linDU} and
Remark~\ref{R-incone} the solution operator $\bq(t;\cdot)$  maps
the set $\coneN\setminus\{0\}$ into its interior $\intconeN$ for
each $t > 0$ maps. Moreover,  relations of the type
\eqref{E-incone} hold for it with appropriate parameters.
\end{remark}

\begin{remark}\label{R-no-inc-simplex}
The relation $\boldsymbol{1}_N^T \Bar{Q} \equiv \boldsymbol{0}$ is
not valid in general. Then the simplex $\Sigma_N$ will not be
  invariant under $\bq(t;\cdot)$, but this will be of no
importance in what follows.
\end{remark}

\begin{theorem}\label{T-main}
Let $\theta$  a metric dynamical system and  let  $Q(\omega)$ be a
matrix  of the tridiagonal form \eqref{E-defQ}, which satisfies
\eqref{E-oneQ2zero2} and \eqref{E-qbound}. Then, the solution
$\bp(t;\omega,\bp_{0})$ of   the random linear differential
equation \eqref{RDE1} satisfies
\begin{equation}\label{E-mainineq}
\bp(t;\omega,\bp_{0})\ge \bq(t;\bp_{0}) \quad \mbox{for all  } \omega\in\Omega,~ t \geq 0,
\bp_{0}\in\coneN,
\end{equation}
where the inequality is meant componentwise.
\end{theorem}

\begin{proof}
Fix an $\omega\in\Omega$ and fix the initial conditions $\bp(0) =
\bq(0) = \bp_{0} \in \coneN$ for the differential equations
\eqref{RDE1} and \eqref{E-QbarDU}, respectively. Then the function
\[
\bx(t):=\bp(t;\omega,\bp_{0})-\bq(t;\bp_{0})
\]
satisfies the differential equation
\[
\frac{d\bx}{dt} = Q(\theta_{t}\omega) \bp-\Bar{Q}\bq=Q(\theta_{t}\omega) \bx+\left(Q(\theta_{t}\omega)-\Bar{Q}\right)\bq
\]
with the initial condition $\bx(0)=0$. Denoting
\[
\bw(t)
:=\left(Q(\theta_{t}\omega)-\Bar{Q}\right)\bq(t;\bp_{0}),
\]
this differential equation can be written as
\begin{equation}\label{E-DUx}
\frac{d\bx}{dt} = Q(\theta_{t}\omega)
\bx+\bw(t),\quad \bx(0)=0.
\end{equation}

The matrix $\left(Q(\theta_{t}\omega)-\Bar{Q}\right)$ has,  by
definition,   non-negative components for every $t$ and $\omega$.
By Remark~\ref{R-bpprop} the  function $\bq(t;\bp_{0})$ with
$\bp_{0}\in\coneN$ has non-negative components for all $t\ge0$.
Hence  the function $\bw(t)$ has  non-negative components.
Moreover, the matrix  $Q(\omega)$ as a matrix with non-negative
off-diagonal components has a path of nonsingularity, so it follows
that the differential equation \eqref{E-DUx} satisfies all the
conditions of Theorem~\ref{T-diffpos}. Hence, the solution function
$\bx(t)$ is positive, i.e., its components are non-negative.
\end{proof}

Theorem~\ref{T-main} has some important  consequences.

\begin{corollary}\label{C1} There following statements are valid:
\begin{enumerate}[\rm(i)]

\item for any $\omega\in\Omega$ and $t>0$, the solution
    operator $\bp(t;\omega,\cdot)$ maps the set
    $\coneN\setminus\{0\}$ into its interior $\intconeN$;

\item for any $\omega\in\Omega$ and $t>0$, the solution
    operator $\bp(t;\omega,\cdot)$ maps the simplex $\Sigma_N$
    into its interior $\intsimpN$;

\item for any bounded interval $[T_{1},T_{2}]\subset(0,\infty)$
    there is a number $\beta(T_{1},T_{2})$ $>$ $0$ such that
    \[
    \bp(t;\omega,\bp_{0})+\beta(T_{1},T_{2})\|\bp_{0}\| \, \left(\mathbb{S}^{N}\cap\Sigma_N\right)
    \subseteq\Sigma_N,
    \]
for all $\omega\in\Omega$, $\bp_{0}\in\coneN\setminus\{0\}$,
$t\in[T_{1},T_{2}]$.
\end{enumerate}
\end{corollary}

To prove Corollary~\ref{C1} it suffices to note that Assertion (i)
follows immediately  from inequality \eqref{E-mainineq} and
Remark~\ref{R-incone}. This, due to \eqref{E-oneQ2zero2}, implies
Assertion (ii). Both these Assertions together with the inclusion
\eqref{E-incone} valid for the solution operator $\bq(t;\cdot)$
imply Assertion (iii).

\section{Random dynamical systems}\label{S-random}

Let $(\Omega,\mathcal{F},\mathbb{P})$ be a probability space. A
random dynamical system \cite{ArnoldL:98,Chueshov:02} has a skew
product structure consisting of a  metric dynamical system
\eqref{Def-MDS} that models the driving noise process and a cocycle
mapping.

Let $M$ be a complete metric space equipped with the metric $\rho$
and let $\theta$ be a metric dynamical system.

\begin{definition}\label{D-cocycle}
A map $\varphi : \mathbb{R}^+\times\Omega\times M \to M$ is called
a cocycle on $M$ with respect to the driving system $\theta$ if it
satisfies the \emph{initial condition}
\[
\varphi(0,\omega,\cdot)= \mbox{id}_X \quad \mbox{for all  } \omega\in\Omega,
\]
and the \emph{cocycle property}
\[
\varphi(t+s,\omega,x) =
\varphi(s,\theta_{t}\omega,\varphi(t,\omega,x)) \quad \mbox{for all  }
x\in M,~\omega\in\Omega,~ t,s\in\mathbb{R}^+.
\]
\end{definition}

The RDE \eqref{RDE} generates a random dynamical system
$(\theta,\varphi)$ on $\mathbb{R}^{N}$ with the cocycle mapping
$\varphi$ defined by
\[
\varphi(t, \omega, \bp_{0}) = \bp(t, \omega; \bp_{0})\quad \mbox{for all  }
t \geq 0,~\omega\in\Omega,~\bp_{0}\in\mathbb{R}^{N}.
\]
In fact,  the mapping $(t, \bp) \mapsto \varphi(t,\omega, \bp)$ is
continuous for each $\omega\in\Omega$, while the mapping $\bp$
$\mapsto \varphi(t,\omega, \bp)$ is linear.

\subsection{Random attractors}\label{S-randatt}

The asymptotic behaviour of a random dynamical system is described
by random attractors.
 \begin{definition}\label{ranattr}
A random attractor $\mathscr{A} = \{A(\omega), \omega\in\Omega\}$
of a random dynamical system $(\theta,\varphi)$ is a family of
nonempty compact subsets of $M$  such that
\begin{enumerate}[\rm 1.]
 \item the setvalued mapping $ \omega \mapsto A(\omega)$ is
     $\mathcal{F}$--measurable.

\item  $\mathscr{A}$ is $\varphi$-invariant, i.e.,
\[
\varphi\left(t,\omega,A(\omega)\right) = A(\theta_t \omega)
\]
for every $\omega \in \Omega$ and $t \in \mathbb{R}^+$,

\item  $\mathscr{A}$ pullback attracts
    $\mathcal{F}$--measurable families  $\mathscr{D} =$
    $\{D(\omega), \omega\in\Omega\}$  of  nonempty compact
    subsets of $M$, i.e.,
\[
\mbox{\rm dist} \left(\varphi\left(t,\theta_{-t} \omega, D\left(\theta_{-t} \omega\right)\right), A(\omega)\right) \longrightarrow 0 \quad \mbox{as  }  t \to 0
\]
for every $\omega \in \Omega$
\end{enumerate}
\end{definition}

It follows from Theorem~\ref{T-main} and Corollary~\ref{C1}  that
the solution operator $\bp(t;\omega,\cdot)$ of the  RDE \eqref{RDE}
maps the simplex $\Sigma_N$ into its interior $\intsimpN$
uniformly for  $t$ in bounded intervals from $(0,\infty)$.

The \emph{Hilbert projective metric} $\rho_{H}$ on $\coneN$ will be
used. It can be  defined as
\[
\rho_{H}(x,y)=\left|\ln\left(\frac{\max_{i}y_{i}/x_{i}}{\max_{i}x_{i}/y_{i}}\right)\right|
\]
for vectors $x = (x_{1},x_{2},\ldots,x_{N})^\top$ and $y =$
$(y_{1},y_{2},\ldots,y_{N})^\top$ in $\coneN$. It is fact only a
semi-metric on  $\coneN$, but becomes a metric on a projective
space. Important here is that the interior $\intsimpN$ of the
probability simplex $\simpN$ is the complete metric space with the
Hilbert projective metric, see, e.g.,
\cite{KLS:PosLinSys:e,Nuss:SIAMJMA90,Nuss:DIE94} and references
therein on the properties of the Hilbert projective metric. In the
Russian literature the Hilbert projective metric is also called the
Birkhoff metric.

The main result of this paper, Theorem \ref{T-attr} below, requires
the following definitions.

\begin{definition}
A cocycle mapping $\varphi : \mathbb{R}^{+}\times\Omega\times M$
$\to M$, where $(M,\rho)$ is a metric space,  is called
\emph{uniformly dissipative} if there exist a number $T_{d} \in$
$\mathbb{R}^{+}$ and a closed bounded set $M_{0} \subset M$ such
that
\[
\varphi(T_{d},\omega,M)\subseteq M_{0},\quad \forall~ \omega\in\Omega.
\]
It  is called \emph{uniformly contractive} if there exist a number
$T_{c} \in \mathbb{R}^{+}$ and a number $\lambda \in [0,1)$ such
that
\[
\rho(\varphi(T_{c},\omega,x),\varphi(T_{c},\omega,y))\le\lambda \rho(x,y),\quad
\forall~ \omega\in\Omega, ~x,y\in M.
\]
\end{definition}

\begin{theorem}\label{T-attr}
The restriction of $\bp(t;\omega,\cdot)$ to the set $\simpN$ is a
uniformly dissipative and uniformly contractive cocycle (with
respect to the Hilbert projective metric), which has a random
attractor $\mathscr{A} = \{A(\omega), \omega\in\Omega\}$ such that
set $A(\omega) = \{a(\omega)\}$ consists of a single point for each
$\omega \in \Omega$. Moreover, the random attractor is
asymptotically stable, i.e.,
\[
\left\|\bp(t,\omega,\bp_0) - a(\theta_t\omega)\right\| \to 0 \quad \textrm{as}\quad t \to \infty
\]
for all $\bp_0 \in \intsimpN$ and $\omega \in \Omega$.
\end{theorem}

\begin{proof}
The proof  follows by an application of the discrete-time case in
\cite{KloKoz:DCDSB11} to, e.g., the time-one mapping of the RDE,
i.e., the cocycle mapping at integer time values. The uniform
dissipativity of the cocycle ensures the existence of a random
attractor, while the uniform contractivity implies that the
components subsets are singleton sets.
\end{proof}


  \providecommand{\nosort}[1]{} \providecommand{\bbljan}[0]{January}
  \providecommand{\bblfeb}[0]{February} \providecommand{\bblmar}[0]{March}
  \providecommand{\bblapr}[0]{April} \providecommand{\bblmay}[0]{May}
  \providecommand{\bbljun}[0]{June} \providecommand{\bbljul}[0]{July}
  \providecommand{\bblaug}[0]{August} \providecommand{\bblsep}[0]{September}
  \providecommand{\bbloct}[0]{October} \providecommand{\bblnov}[0]{November}
  \providecommand{\bbldec}[0]{December}

\end{document}